\documentclass{article}

\usepackage{geometry}
\usepackage{amsfonts}
\usepackage{amssymb}
\usepackage{amsmath}
\usepackage{color}
\usepackage{fancyhdr}
\usepackage{amsthm}
\usepackage{hyperref}
\usepackage{bookmark}
\usepackage{verbatim}
\usepackage{float}
\usepackage[all]{xy}
\usepackage{graphicx}
\usepackage{caption}
\usepackage{subcaption}
\usepackage{mathtools}
\usepackage{youngtab}
\usepackage{tikz}

\DeclarePairedDelimiter{\ceil}{\lceil}{\rceil}

\newtheorem{theorem}{Theorem}[section]
\newtheorem{question}[theorem]{Question}
\newtheorem{lemma}[theorem]{Lemma}
\newtheorem{proposition}[theorem]{Proposition}
\newtheorem{algorithm}[theorem]{Algorithm}
\newtheorem{corollary}[theorem]{Corollary}
\newtheorem{conjecture}[theorem]{Conjecture}
\newtheorem{fact}[theorem]{Fact}
\newtheorem{problem}[theorem]{Problem}
\newtheorem{definition}[theorem]{Definition}
\newtheorem{example}[theorem]{Example}

\title{Distinguishing Numbers and Generalizations}
\author{Caleb Ji}
\date{July 30, 2018}

\begin{document}
\vspace{-3cm}
\maketitle

\begin{abstract}
The distinguishing number of a graph was introduced by Albertson and Collins in~\cite{Alb1} as a measure of the amount of symmetry contained in the graph.  Tymoczko extended this definition to faithful group actions on sets in~\cite{Tymoczko}; taking the set to be the vertex set of a graph and the group to be the automorphism group of the graph allows one to recover the previous definition.  Since then, several authors have studied properties of the distinguishing number as well as extensions of the notion.  In this paper, we first answer a few open questions regarding the distinguishing number.  Next we turn to generalizations regarding the labeling of Cartesian powers of a set and the different subgroups that can be obtained through labelings.  We then introduce a new partially ordered set on partitions that follows naturally from extending the theory of distinguishing numbers to that of distinguishing partitions.  Then we investigate the groups obtainable from partitioning Cartesian powers of a set in more detail and show how the original notion of the distinguishing number of a graph can be recovered in this way.  Next, we introduce a polynomial and a symmetric function generalization of the distinguishing number.  Finally, we present a large number of open questions and problems for further research.
\end{abstract}

\section{Introduction}
The distinguishing number of a graph is one way of measuring the amount of symmetry it possesses.  Albertson and Collins defined an $r$-distinguishing labeling of a graph $G$ to be a function $\phi: V(G)\rightarrow \{1, 2, \ldots, r\}$ such that the only element $f\in \operatorname{Aut}(G)$ that preserves the labels is the identity~\cite{Alb1}.  The distinguishing number is defined as the smallest $r$ for which such a distinguishing labeling exists.  For instance, if $G=K_n$, then a different label is required for each vertex, and $D(G)=n$.  If $G$ has trivial automorphism group, then $D(G)=1$.  Furthermore, the distinguishing number of a graph is equal to that of its complement. \\

A classical example is that of the cycle graph on $n$ vertices, which has automorphism group $D_n$.  Because it is easy to visualize the symmetries, one can imagine how $3$ labels are necessary to remove all the symmetries for $n=3,4,5$, but that $2$ labels suffice for larger $n$.

\begin{figure}[H]
\label{ex1}

\begin{subfigure}{.5\textwidth}
\begin{tikzpicture}
\draw[fill=black] (0,0) circle (3pt);
\draw[fill=black] (3,0) circle (3pt);
\draw[fill=black] (4,2.8) circle (3pt);
\draw[fill=black] (1.5,4.75) circle (3pt);
\draw[fill=black] (-1,2.8) circle (3pt);
\node at (-0.25,0) {1};
\node at (3.25,0) {2};
\node at (4.25,2.8) {3};
\node at (1.75,4.75) {1};
\node at (-1.25,2.8) {2};
\draw (0,0) -- (3,0) -- (4,2.8) -- (1.5,4.75) -- (-1,2.8) -- (0,0);
\end{tikzpicture}
\end{subfigure}
\begin{subfigure}{.5\textwidth}
  \begin{tikzpicture}
\draw[fill=black] (0,0) circle (3pt);
\draw[fill=black] (3,0) circle (3pt);
\draw[fill=black] (4.5,2.6) circle (3pt);
\draw[fill=black] (3,5.2) circle (3pt);
\draw[fill=black] (0,5.2) circle (3pt);
\draw[fill=black] (-1.5,2.6) circle (3pt);
\node at (-0.25,0) {1};
\node at (3.25,0) {1};
\node at (4.75,2.6) {2};
\node at (3.25,5.2) {2};
\node at (-0.25,5.2) {1};
\node at (-1.75,2.6) {2};
\draw (0,0) -- (3,0) -- (4.5,2.6) -- (3,5.2) -- (0,5.2) -- (-1.5,2.6) -- (0,0);
\end{tikzpicture}
\end{subfigure}
\caption{Minimal distinguishing labelings for the cycle graph for $n=5, 6$}
\end{figure}
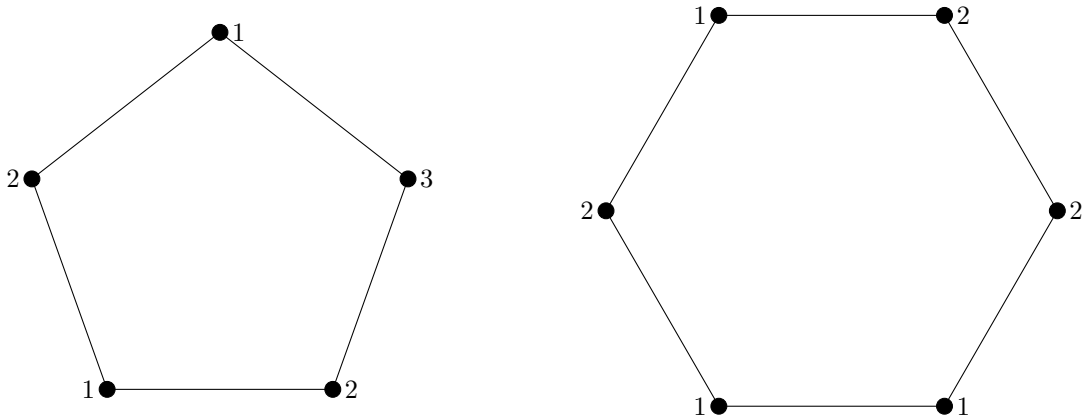

Tymoczko extended this definition to arbitrary faithful group actions in~\cite{Tymoczko}.  Taking any group $G$ to act on a set $X$, a distinguishing labeling of $X$ with respect to the action of $G$ is defined to be one such that only the identity element in $G$ preserves the labels of $X$.  The action is required to be faithful, or else there will always be a non-identity element that acts trivially on the set. The case of a graph is obtained when $X$ is taken to be the vertex set of a graph and $G$ is taken to be its automorphism group.  This more general case yields different results, as not all subgroups of $S_n$ can be realized as an automorphism group of a graph with $n$ vertices.  For instance, there is a faithful action of $S_4$ with distinguishing number 3, but no graph with automorphism group $S_4$ and distinguishing number 3~\cite{Tymoczko}. \\

A large body of research has been devoted to studying properties of the distinguishing number~\cite{Alb2, Bog, Chan1, Chan2, Imrich}.  For instance, in~\cite{Sharma}, the authors investigated the distinguishing numbers of the wreath product groups $\overrightarrow{S_n}$ and  $\overrightarrow{A_n}$.  There are several other interesting avenues for research on this topic; for example, in~\cite{Bird}, the authors presented a new algorithm that generates all distinguishing labelings.  In this paper, we begin by resolving a few conjectures made by previous authors on this topic.  These lead into several natural generalizations of this topic which we define and provide the basic theory.  These lead to many new directions that can be explored. \\

The organization of this paper is as follows.  In Section~\ref{dcg}, we resolve a conjecture regarding so-called distinguishing critical graphs.  In Section~\ref{sn}, we answer a question of Tymoczko in~\cite{Tymoczko}.  In Section~\ref{gen1}, we consider two generalizations of the problem at hand, namely involving multiple labelings and obtaining subgroups other than the identity after making the partitions. In Section~\ref{poset}, we generalize the notion of the distinguishing number to that of a distinguishing partition, and introduce a new partial ordering on the partitions of $n$ that naturally follows from it.  In Section~\ref{gen}, we present another generalization of all these ideas which opens the way for techniques from representation theory to study the more general problems.

\section{Distinguishing critical graphs}
\label{dcg}
In~\cite{Ali}, Alikhani and Soltani introduced the notion of a distinguishing critical graph.  We reproduce their definition here:

\begin{definition}
[Distinguishing critical graph] A distinguishing critical graph $G$ is one such that for every induced subgraph $H$ of $G$, $D(H)\neq D(G)$.
\end{definition}

Note that an induced subgraph $H$ of $G$ may have a larger distinguishing number than $G$.  For example, if $H$ is the complete graph on 6 vertices and $G$ consists of $H$ and another vertex that is connected to three vertices from $H$, then $D(H)=6$ and $D(G)=3$. \\

Alikhani and Soltani conjectured that a distinguishing critical graph is regular and checked this for $D(G)=1,2,3$.  Here we prove the following stronger result
  
\begin{theorem}
\label{crit}
A distinguishing critical graph has a transitive automorphism group.
\end{theorem}

This theorem implies the conjecture because if vertices $v_1$ and $v_2$ are in the same orbit under $\operatorname{Aut}(G)$, they must have the same degree.  The key behind the proof comes from the idea of considering the orbits of the vertices, which was done by Tymoczko in~\cite{Tymoczko}.  The point is that two elements that are in different orbits can always be given the same label, since there is no element of $\operatorname{Aut}(G)$ that will bring one to the other.  From this fact, the following lemma in~\cite{Tymoczko} is clear:

\begin{lemma}
\label{l1}
\cite{Tymoczko}
Fix an orbit $O$ under the action of a group $\Gamma$ on a set $X$. Let $\phi_1$ be a $k_1$-distinguishing
labeling of $O$ under the action of $\phi_1$ and let $\phi_2$ be a $k_2$-distinguishing labeling of $X\backslash O$ under
the action of $\phi_2$. The labeling $\phi$ defined by $\phi_O = \phi_1$ and $\phi_{X\backslash O} = \phi_2$ is a $\max\{k_1, k_2\}$-
distinguishing labeling of $X$ under the action of $\Gamma$.
\end{lemma}

We will need one more easy lemma for our proof.

\begin{lemma}
\label{l2}
Let $H$ be the induced subgraph of $G$ obtained by deleting a single vertex.  Then $D(H)\ge D(G)-1$.
\begin{proof}
Label $H$ with $D(H)$ labels and label the remaining vertex of $G$ with $D(H)+1$.  Then any automorphism of $G$ that preserves the labels will fix that last vertex because it has its own label, and will also fix the vertices of $H$ because it must induce an automorphism of $H$.  Thus this is a distinguishing labeling for $G$ and thus $D(G) \le D(H)+1$.
\end{proof}
\end{lemma}

\begin{proof}[Proof of Theorem~\ref{crit}]
Assume that $G$ is critical but that $\operatorname{Aut}(G)$ is not transitive.  Then there is an orbit $O_1$ of the vertices that does not include all vertices of $G$.  By Lemma~\ref{l1}, $D(G)\le \max \{D(O_1), D(G\backslash O_1)\}$.  This implies that, taking $H$ to be either $O_1$ or $G\backslash O_1$ (whichever has larger distinguishing number), $D(G)\le D(H)$ for some induced subgraph $H$ of $G$.  Then by Lemma~\ref{l2}, we may remove vertices one at a time from $H$ until the distinguishing number of the resulting graph equals that of $G$.  Thus $G$ can only have one orbit if it is distinguishing critical, as desired.
\end{proof}

\section{Distinguishing numbers for \texorpdfstring{$S_n$}{S_n}}
\label{sn}
\subsection{Overview of methods}
Let $\Gamma$ be a finite group acting on a set $X$ with $n$ elements.  In~\cite{Tymoczko}, Tymoczko gives an algorithm to construct a distinguishing label that takes at most $k$ labels if $|\Gamma| \le k!$.  We reproduce the algorithm here.  In this algorithm and for the remainder of the paper, $\operatorname{Stab}(X_1)$ refers to the pointwise stabilizer of a set given an implied group action.

\begin{algorithm}
\label{alg}
\cite{Tymoczko}
\begin{enumerate}
\item Initialize $i=1$ and set $\phi(x)=1$ for all $x$ in $X$.  Let $\Gamma_1=\Gamma$ and $X_1=X$.  
\item While $\Gamma_i\neq \operatorname{Stab}(X_i)$ do 
\begin{itemize}
\item [(a)] Choose a subset $X'_{i+1}$ of $X_i$ that contains a unique element from each nontrivial $\Gamma_i$-orbit in $X_i$, namely so that the intersection $|X'_{i+1}\cap \Gamma_ix| = 1$ for each $x$ in
$X_i$ such that $\Gamma_ix$ has at least two elements.
\item [(b)] Label the elements of $X'_{i+1}$ with $i+1$, so $\phi(x)=i+1$ for each $x$ in $X'_{i+1}$.
\item [(c)] Let $X_{i+1}=X_i\backslash X'_{i+1}$ and let $\Gamma_{i+1}=\operatorname{Stab}_{\Gamma_i}(X'_{i+1})$.
\item [(d)] Increment $i$ by 1. 
\end{itemize}
\end{enumerate}
\end{algorithm}

The key idea behind this algorithm is the fact that if $\Gamma$ acts transitively on some set of size $m$, then by the orbit-stabilizer theorem, the size of the subgroup that fixes one of those elements is $\frac{|\Gamma|}{m}$.  Based on this principle, Tymoczko proved the following theorem.

\begin{theorem}
\cite{Tymoczko} 
If $|\Gamma|\le k!$, then $D_{\Gamma}(X) \le k$.
\end{theorem}
This is proven by simply applying Algorithm~\ref{alg} and noting how the size of the automorphism group that preserves the labels decreases.

The following lemma will prove to be useful.
\begin{lemma}
\label{stab1}
The pointwise stabilizer of an orbit is a normal subgroup.
\begin{proof}
Note that the pointwise stabilizer of an orbit is the kernel of the homomorphism induced by the action sending the group to the automorphism group of the orbit.  The result follows from the fact that the kernel of a homomorphism is always a normal subgroup.
\end{proof}
\end{lemma}

\subsection{Distinguishing number \texorpdfstring{$n-1$}{n-1}}
Tymoczko asked if there exist faithful actions of $S_n$ on a set $X$ with $D_{S_n}(X)=n-1$ for arbitrarily large $n$.  Here we answer this question in the affirmative.

\begin{theorem}
\label{con}
For $n\ge 3$, let $|X|=n+2$ with $X = \{x_1, x_2, \ldots, x_{n+2}\}$.  Consider the following action of $S_n$ on $X$: considered as a permutation of $[n]$, an element of $S_n$ acts on $\{x_1, x_2, \ldots, x_n\}$ in the natural way, and permutes $x_{n+1}$ and $x_{n+2}$ if and only if it is an odd permutation.  We claim that the distinguishing number for this action is $n-1$.  
\end{theorem}

\begin{proof}

First we demonstrate an $n-1$ distinguishing labeling.  Label two elements of $\{x_1, x_2, \ldots, x_n\}$ with $1$ and the rest with $2, 3, \ldots, n-1$.  Label $x_{n+1}$ with 1 and $x_{n+2}$ with 2.  Then if a non-identity element of $S_n$ preserves the labels, it must be the permutation $(12)$.  But this is an odd permutation and thus switches $x_{n+1}$ and $x_{n+2}$, which is not allowed.

Now assume there is an $n-2$-distinguishing labeling of $X$.  Then either $3$ elements of $\{x_1, x_2, \ldots, x_n\}$ are given the same label, or there are two distinct pairs which are given the same label.  In either case, there is a non-identity even permutation that preserves both $\{x_1, x_2, \ldots, x_n\}$ and $\{x_{n+1}, x_{n+2}\}$, which is a contradiction.
\end{proof}

\subsection{Further analysis}
Here we analyze all possible faithful actions of $S_n$ with distinguishing number $n-1$.  We begin with a few definitions and lemmas.

Consider a slight variant of Algorithm~\ref{alg} in the following way.  Namely, we finish labeling one pass of an orbit completely before moving onto another.  So in step $1$, we label every element of $O$ with $1$ and in step $2$ we choose one of them to label $2$.  Then in step 3 we only look at one of the orbits that takes $n-2$ labels (including the label 1) to distinguish, label one of them with 3, and then only look at the new orbits within that previous orbit that take $n-3$ labels.  After we completely finish one orbit we can go back and do the rest in similar fashion.  Compared to Algorithm~\ref{alg}, this is analogous to a depth-first-search rather than a breadth-first search.  The reason we do it this way is so that we have better control over the factors by which the subgroup that preserves the labels decreases by at each step. \\

Our methods will rely on the following fact, which follows from the orbit-stabilizer theorem:

\begin{fact}
\label{os}
Let $\Gamma$ act faithfully on a set $X$.  Take some $x\in X$ and let $\Gamma_x$ be the subgroup of $\Gamma$ that preserves $x$.  Then $\frac{|\Gamma|}{|\Gamma_x|} = |\operatorname{Orb}_{\Gamma}(x)|$.
\end{fact}

Fact~\ref{os} tells us that if we decide to label an element of a set with a new color, the factor by which the number of group elements that preserves the labels of the set decreases is equal to the size of the orbit of that element before it was assigned a new label. 

The following definition will prove to be useful.
\begin{definition}
[pseudoclique]
A \textit{pseudoclique} is a subset $S\subseteq X$ such that an element of $S_n/\operatorname{Stab}(X)$ permutes any pair of elements in $S$ while stabilizing every other element of $X$.  
\end{definition}

\begin{lemma}
\label{l4}
If by an application of the stated variant of Algorithm~\ref{alg}, the orbit sizes are $k, k-1, \ldots, 2, 1$, then those $k$ elements of the first orbit form a pseudoclique.
\end{lemma}

\begin{proof}
We use induction on $k$.  The base cases are easy.  Now assume it holds for some $k=1 ,2, \ldots, i$.  Then for $k=i+1$, say the algorithm is applied in the order $x_{i+1}, x_i, \ldots, x_1$.  Then by the inductive hypothesis, $x_1, \ldots, x_i$ form a pseudoclique.  Next, consider stabilizing $x_i$ first and then $x_{i+1}$ rather than the other way around.  Since $x_1, \ldots x_{i-1}$ form a pseudoclique, if $x_{i+1}$ does not form a pseudoclique with $x_1, \ldots x_{i-1}$, then $\operatorname{Stab}(x_{i+1})\subset \operatorname{Stab}(x_i)$.  But then looking at the factors by which the group decreases by, we see that we must have $\operatorname{Stab}(x_{i+1})=\operatorname{Stab}(x_i)$, which is a contradiction.  Since we can obtain any permutation of $x_1, \ldots x_i$ that leaves $x_{i+1} $ constant and any permutation of $x_1, \ldots, x_{i-1}, x_{i+1}$ that leaves $x_i$ constant, we can get any permutation of $x_1, \ldots, x_{i+1}$ as desired.
\end{proof}

\begin{theorem}
\label{nset}
If $S_n$ acts faithfully on a set $X$ with distinguishing number $n-1$, then there must be a subset of $X$ of size $n$ for which $S_n$ acts via all possible permutations.
\end{theorem}

\begin{proof}
Throughout this proof, when we say $S_n$, we will always implicitly refer to the implied action of $S_n$ on the set $X$.  

If $D_{S_n}(X)=n-1$, then that there must be an orbit $O$ that takes at least $n-1$ labels to distinguish under the action of $S_n/\operatorname{Stab}(O)$.  

We claim that all possible applications of the stated variant of Algorithm~\ref{alg} will take precisely $n-1$ labels to finish one pass.   It is not possible for it to take more than $n$ elements.  Otherwise, applying the algorithm to this orbit would take more than $n$ steps, which would imply that $|S_n/\operatorname{Stab}(O)|$ is at least the product of $n+1$ distinct positive integers, but $|S_n|=n!$.  If it takes $n$ steps, then by Lemma~\ref{l4} we end up with $S_n$ acting naturally on an $n$-element subset of $X$, and we have already seen a construction for this in Theorem~\ref{con}.  Thus we can assume $O$ takes $n-1$ labels to distinguish.   Now look at the sequence given by the size of the orbit after each step.  This is a sequence of $n-1$ strictly decreasing integers $a_1, a_2, \ldots, a_{n-1}$ with $a_{n-1}=1$ and with product dividing $n!$.  Take the application of this algorithm that will give the largest number of consecutively increasing numbers beginning from 1. \\ 

We will now discover properties of this sequence $\{a_i\}_{i=1}^{n-1}$, beginning with the following lemma.

\begin{lemma}
\label{l3}
If for some $j$ between $1$ and $n-2$ exclusive, $4 < a_j < \sqrt{n}$, then we cannot have $a_j-a_{j+1}=2$.
\begin{proof}
Assume otherwise.  Call the elements of the orbit associated with the sequence $(a_i)$ $x_1, x_2,\ldots x_k$ before step $j$.  Then say we stabilize some element, say $x_1$ on step $j$, taking the orbit size from $a_j$ to $a_{j+1}=a_j-2$.  This means that another element, say $x_2$, is also stabilized along with $x_1$, and the remaining elements form an orbit.  
Now we claim that the stabilizer of $x_1$ is the same as the stabilizer of $x_2$.  To see this, note that stabilizing both $x_1$ and $x_2$ decreases the group size by a factor of $a_j$.  So if we stabilize $x_2$ first, the resulting orbit of $x_1$ must have size 1 in order to preserve this property.  Thus $\operatorname{Stab}(x_1)\subseteq\operatorname{Stab}(x_2)\subseteq\operatorname{Stab}(x_1)$, form which the desired result follows.

Now if we stabilize $x_3$ in our next step, the size of the resulting subgroup will be $\frac{1}{a_j(a_j-2)}$ of the original.  Note that if we stabilize $x_3$ and then $x_1$ we get the same subgroup.  This implies that after stabilizing $x_3$, $x_1$ must be in an orbit of size $k-2$.  That means some other element must be stabilized by $x_3$; it can't be $x_2$, or else the stabilizer of $x_1$ would also be the same as the stabilizer of $x_3$.  Thus the stabilizer of $x_3$ is the stabilizer of a different element, say $x_4$.  Then in our original sequence, if we stabilize $x_3$ after stabilizing $x_1$, the values of the $(a_i)$ will drop by at least 2 twice in a row, as $a_j > 4$.  Since there are $n-1$ elements in the sequence $\{a_i\}_{i=1}^{n-1}$, the largest element will be at least $n+1$.  Then the product of the $a_i$ divided by $n!$ will be at least $\frac{n+1}{\sqrt{n}^2}>1$, which is a contradiction.
\end{proof}
\end{lemma}

We return to the proof of Theorem~\ref{nset} and split into cases based on the properties of $\{a_i\}_{i=1}^{n-1}$, showing that none of them are possible.

Case 1: $\{a_i\}_{i=1}^{n-1}$ includes 2. 

Let's say that the sequence $(a_i)$ includes $1, 2, \ldots, k$, but not $k+1$.  Then those $k$ elements of $x$, call them $x_1, x_2, \ldots, x_k$, form a pseudoclique.  First consider the case in which $k=n-1$.  Since $a_1=|O|$, then $|O|=n-1$ and the stabilizer of $O$ has size $\frac{|S_n|}{(n-1)!}=n$.  But $\operatorname{Stab}(O)$ is a normal subgroup of $S_n$ by Lemma~\ref{stab1}, and there are no normal subgroups of $S_n$ of size $n$.  Thus $k\neq n-1$, so the values of $(a_i)$ must decrease by at least two at some point.  Now suppose $k<\sqrt{n}-2$.  Then $a_{k+1}\ge a_k+2$, which will make the product of the $(a_i)$ too big for the reason given in the proof of Lemma~\ref{l3}.  Next suppose that $k\ge \sqrt{n}$.  Say that we go from $k+t$ to $k$ by stabilizing an element $x_{k+1}$; let these other $t$ elements be $x_{k+1}, \ldots, x_{k+t}$.  Then instead of stabilizing $x_{k+1}$ and then $x_k$, let's stabilize $x_k$ and then $x_{k+1}$.  Then $x_{k+1}$ has to get into an orbit of size $k$ after $x_k$ is stabilized, or else the rest would only take at most $k-1$ steps.  Since $k\ge t-1$ and $x_1, \ldots, x_{k-1}$ are still in an orbit, $x_{k+1}$ must join that orbit unless $t\ge k-2$ and $k\le n/2$.  In the former case, in order for the sequence to take the correct number of moves, it can only decrease by 1 at a time, meaning $x_{k+1}$ forms a pseudoclique with $x_1, \ldots, x_{k-1}$, and thus with $x_k$ as well.  This contradicts the maximality of $k$.  Thus $t\ge k-2$.  Then the minimal product of the $a_i$ is at least $k!(2k-2)(2k-1)\cdots (n+k-3)$.  Dividing this by $(n-1)!$ gives $\dfrac{(n+k-3)\cdots (2k-2)}{(n-1)\cdots (k+1)}$, which, since $\sqrt{n}\le k\le n/2$, is at least $2^{\sqrt{n}} > n$.  Thus the size of $\Gamma_1$ is greater than $n!$, contradiction. 

Case 2: $\{a_i\}_{i=1}^{n-1}$ does not include $2$, and is not equal to the sequence $\{n, n-1, \ldots, 3\}$.  

We know that $a_{n-1}=3$, or else our product will be at least $4\cdot 5\cdot\cdots \cdot n+1 > n!$.  For the same reason, $\{a_i\}_{i=1}^{n-1}$ must contain consecutive numbers from at least $a_k = n/2$ down through $3$.  Because $A_n$ is the only subgroup of $S_n$ with index $2$, the group must act as $A_{n-k}$ on the orbit corresponding to the term $a_k$.  We know that the sequence must skip a number at some point; then simply apply the same argument as in Case 1 to show that this implies that the size of the original group is greater than $n!$.

Case 3: $\{a_i\}_{i=1}^{n-1} = \{n, n-1, \ldots, 3\}$.

Then there are $n$ elements on which $S_n$ acts as $A_n$.  Then the stabilizer of these $n$ elements is a normal subgroup of size $2$ in $S_n$, which is impossible.
\end{proof}

\section{Multiple labelings and obtaining subgroups}
\label{gen1}
In this section we present two ways to generalize the notion of the distinguishing number.  First, we consider labeling elements of $X$ with tuples rather than single labels.  We can also ask which subgroups of the original group may be obtained given only the restriction that it preserve some fixed labeling.

\subsection{Multiple labelings}
Going back to the construction in Theorem~\ref{con}, it may initially seem strange that if $S_n$ acts via all permutations of an $n$-element subset of $X$, that it may still have a distinguishing number less than $n$.  It is therefore a natural question to ask what the distinguishing numbers are for a group acting on multiple orbits is if there is some known link between how it acts on the orbits.  We present one simple case here and comment on further directions of study in the last section of this paper.

\begin{theorem}
\label{mul}
Let $X$ be a set with $nk$ elements partitioned into $k$ subsets $X_1, X_2, \ldots, X_k$ of size $n$.  Let $S_n$ act on $X$ by acting via every possible permutation on each subset simultaneously.  Then $D_{S_n}(X) = \ceil{\sqrt[k]{n}}$.
\end{theorem}

\begin{proof}
First we show that $\ceil{\sqrt[k]{n}}$ labels suffice by constructing a $t$-distinguishing labeling for $n=t^k$.  For each subset $X_i$, designate a given order on its elements. Then pick a different string of $k$ integers from $1$ through $t$: $(c_1, c_2, \ldots, c_k)$ for each value $1$ through $n$ and label the $k$ elements in each position of the ordering with those labels.  Then every non-identity element of $S_n$ will inevitably send some string of labelings to a different string of labelings, so only the identity can preserve all the labelings.  \\

Now assume that there is a distinguishing labeling using fewer labels.  Then if there are $t$ labels, we have $n>t^k$ and by the Pigeonhole Principle, two elements must share the same string and thus can be exchanged.
\end{proof}

As the previous proof shows, a more natural way of looking at this question is by allowing the labels to be tuples.  To this end, we have the following definition.

\begin{definition}
$D^k_{\Gamma}(X)$ is the smallest number $r$ of labels such that there exists a labeling $\phi: X\rightarrow \{1, 2, \ldots, r\}^k$ with the property that the only element of $\Gamma$ that preserves the labels is the identity.
\end{definition}

Under this notation, if $|X|=n$ and $S_n$ acts on $X$ via all possible permutations, Theorem~\ref{mul} says that $D^k_{S_n}(X) = \ceil{\sqrt[k]{n}}$.

\subsection{Obtaining subgroups of abelian groups}
Here we consider the question of which subgroups can be obtained through labelings for abelian groups.

\begin{theorem}
\label{abe}
Let $G$ be a finite abelian group that acts faithfully on a set $X$.  Then for any subgroup $H\le G$, there is a labeling $\lambda$ with two colors such that $G\cap P_{\lambda} = H$.

\begin{proof}
Pick a set $\{x_j\}$ of representatives of the orbits of the action of $G$ on $X$ as $j$ indexes these orbits.  Then label all elements of the form $h(x_j)$ in orbit $j$ with 1, where $h\in H$.  Label all other elements with 2.  

First we show that all $h\in H$ preserves these labels.  Indeed, if $x\in X$ is in the $H$-orbit of some $x_j$, then $h(x)$ is as well, and both are labeled 1.  Otherwise, $x$ and $h(x)$ are both labeled 2.  Thus $x$ and $h(x)$ are always given the same label.

Now assume $g\in G$ preserves the labels; we must show that $g\in H$.  We claim that $g$ is identified with an element of $H$ under the projection $G/\operatorname{Stab}(O_j)$ for each orbit $O_j$.  Indeed, let $h_j$ be an element of $H$ such that $h_j(x_j)=g(x_j)$.  Then for any $x_j'\in O_j$, there is some $g'\in G$ such that $g'(g(x_j))=x_j'$. Then $g(x_j')=g(g'(h_j(x_j)))=h_j(g'(g(x_j))) = h_j(x_j')$.  Therefore $gh_j^{-1}\in \operatorname{Stab}(O_j)$, as desired.  Now consider the image of $g$ under the projection map $G\rightarrow G/H$.  We have shown that it must be in the intersection of the projections of $\bigcap_j \operatorname{Stab}(O_j)$.  But since $G$ acts faithfully on $X$, $\bigcap_j \operatorname{Stab}(O_j)=\{e\}$.  Thus $g\in H$, as desired. 
\end{proof}
\end{theorem}

\section{The consumption ordering of partitions}
\label{poset}
\subsection{Introduction}
Let $\Gamma$ act on $X$.  Then given any labeling $\phi$ of $X$, we can define the type of $\phi$ to be the partition of $|X|$ with $i^\text{th}$ part (in decreasing order of size) the number of elements of $X$ given the $i^\text{th}$ most common label.  Then given a labeling $\phi$ of type $\lambda$, the permutations of $X$ that preserve the labelings are given by $P_{\phi}=S_{\lambda_1}\times S_{\lambda_2}\times\cdots\times S_{\lambda_k}$, where $S_{\lambda_i}$ acts on the elements.  The intersection of the image of $\Gamma$ in $\operatorname{Sym}(X)$ and $P_{\phi}$ is the subgroup of $\Gamma$ that preserves the labelings.  Henceforth for convenience, we will use $\Gamma$ to also represent its image in $\operatorname{Sym}(X)$ if $\Gamma$ acts on $X$.

The distinguishing number of $\Gamma$, then, is simply the shortest length of a partition $\lambda$ such that there is a labeling $\phi$ of $X$ with type $\lambda$ such that $\Gamma\cap P_{\phi}=\{e\}$.  

While the distinguishing number gives information on the number of labels necessary to distinguish a group action, it doesn't give any more information about the partition of the labels itself.  Studying these partitions themselves is a natural extension.

\begin{definition}
[distinguishing partition]
Let $H$ be a subgroup of $S_n$.  A set partition $[n] = A_1\cup\cdots\cup A_k$ is said to distinguish $H$ if the only element of $H$ that preserves each of the subsets in the set partition is the identity.  The set partition itself is called a \textit{distinguishing set partition}, while a \textit{distinguishing partition} refers to a partition of $n$ for which there exists a distinguishing set partition.
\end{definition}

\begin{definition}
[consumes] A subgroup $H<S_n$ \textit{consumes} a partition $\lambda\vdash n$ if there is a set partition of $[n]$ with weight $\lambda$ that distinguishes $H$.
\end{definition}

Note that the previous definition simply says that if there is a distinguishing labeling corresponding to a partition $\lambda\vdash n$ with respect to the action of $H$ on $[n]$, then $H$ consumes $\lambda$.  This induces a poset structure on partitions of $n$ with the following ordering relation: $\lambda \ge_c \mu$ if every $H$ that consumes $\lambda$ also consumes $\mu$.  We call this ordering the \textit{consumption} ordering.  It is not immediate that this defines a partial ordering, for it is conceivable that two different partitions are consumed by the same subgroups of $S_n$.  This would violate the condition of antisymmetry.  We show that this does not happen in the next theorem.

\subsection{Structure of the poset}
Recall the dominance ordering on partitions:

\begin{definition}
[dominance order]
Let $\lambda$ and $\mu$ be two partitions of $n$ with parts in decreasing order.  Under the dominance ordering, $\lambda\ge \mu$ if and only if $\sum_{i=1}^k\lambda_i\ge \sum_{i=1}^k\mu_i$ for all $i$ where the expressions are defined.
\end{definition}

We show that the consumption ordering is indeed a partial ordering by showing it is consistent with the dominance ordering.  We use $\ge_c$ to denote the consumption ordering and $\ge$ to denote the dominance ordering.

\begin{theorem}
\label{dom}
If $\lambda\ge_c \mu$, then $\lambda\ge \mu$.
\end{theorem}

\begin{proof}
Assume otherwise.  Then let $\lambda$ and $\mu$ be different partitions of $n$ such that $\lambda\ge_c\mu$ and $\lambda\not\ge\mu$.  Let $\lambda^\intercal = (a_1, a_2, \ldots, a_k)$ and take a labeling $X$ of a Young diagram of $\lambda$.  Then let $H = S_{a_1}\times\cdots\times S_{a_k}$ where each $S_{a_i}$ acts on the $i$th column of the labelled Young diagram.  Then this consumes $\lambda$ because the only intersection of $H$ with the subgroup that permutes the rows is the identity element.  

We claim that $H$ does not consume $\mu$.  Take any labeling $Y$ of $\mu$.  If $H$ consumes $\mu$ under this labeling, then all the the elements of every column in the labeling $X$ of $\lambda$ must be in different columns of $Y$.  Taking the transpose, this tells us that $\mu^\intercal\ge \lambda^\intercal\Rightarrow \lambda\ge\mu$, a contradiction.  Thus if every subgroup that consumes $\lambda$ also consumes $\mu$, then $\lambda\ge\mu$, as desired.
\end{proof}

\begin{corollary}
\label{well}
The consumption ordering is a well-defined poset.
\end{corollary}
\begin{proof}
Reflexivity and transitivity are clear.  For antisymmetry, note that if $\lambda\ge_c\mu$ and $\mu\ge_c\lambda$, by Theorem~\ref{dom} we have $\lambda\ge\mu$ and $\mu\ge\lambda$, so $\lambda=\mu$ as desired.
\end{proof}

By Corollary~\ref{well} we may use the symbol $>_c$ rather than $\ge_c$ to denote this ordering relation.

The converse of Theorem~\ref{dom} is not true.  The following example illustrates this.
\begin{example}
\label{31}
$(3,1)\not >_c(2,2)$.
\end{example}
\begin{proof}
Consider the normal Klein four subgroup of $S_4$: $H=\{(), (1,2)(3,4), (1,3)(2,4), (1,4)(2,3)\}$.  Then $H$ consumes the partition $(3, 1)$, because the stabilizer of any element is the identity.  However, given any labeling of the partition $(2,2)$, interchanging the two columns always corresponds to an element in $H$.  Thus $H$ does not consume $(2,2)$.
\end{proof}

\begin{example}
\label{32}
$(3, 2) > (3,1,1)$.
\end{example}

\begin{proof}
The following diagram gives a ``proof without words."

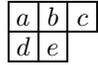
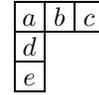
\begin{figure}[H]
\label{wo}
\begin{subfigure}{.5\textwidth}
  \centering
  $\young(abc,de)$
  
  \label{fig:sub1}
\end{subfigure}%
\begin{subfigure}{.5\textwidth}
  \centering
  $\young(abc,d,e)$
  
  \label{fig:sub2}
\end{subfigure}
\end{figure}
\end{proof}

The last example illustrates a general principle that can be used to find ordering relations within the consumption poset.  Fix a labeling of Young diagrams to simply go from $1$ through $n$ from left to right, in each row in order.  If the row subgroup of one partition completely contains that of another one, then the first partition must be greater than the second in the consumption ordering.  However, this principle cannot be used to find all ordering relations.  The following example can be checked to be true by a computer, but not by the principle described above.

\begin{example}
\label{411}
$(4,1,1) >_c (2,2,2)$.
\end{example}

The following figures show the entire poset for $n=4$ and $n=6$.  
\begin{figure}[H]
\label{4}
\centering
\begin{tikzpicture}
    \node (4) at (0, 0) {${\yng(4)}$};
    \node (31) at (-2, -2) {$\yng(3,1)$};
    \node (22) at (2, -2) {$\yng(2,2)$};
    \node (211) at (0, -4) {$\yng(2,1,1)$};
    \node (1111) at (0, -6.5) {$\yng(1,1,1,1)$};
    
    \draw (4) -- (31) -- (211) -- (1111);
    \draw (4) -- (22) -- (211);
\end{tikzpicture}
\caption{The consumption ordering for $n=4$}
\end{figure}
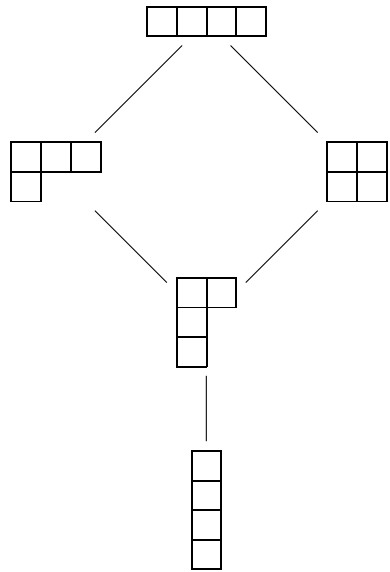

\begin{figure}[H]
\label{6}
\centering
\begin{tikzpicture}
    \node (6) at (0, 0) {$\yng(6)$};
    \node (51) at (0, -2) {$\yng(5,1)$};
    \node (42) at (-4, -3) {$\yng(4,2)$};
    \node (33) at (4, -3) {$\yng(3,3)$};
    \node (411) at (-4, -5) {$\yng(4,1,1)$};
    \node (321) at (-4, -7) {$\yng(3,2,1)$};
    \node (3111) at (-4, -9) {$\yng(3,1,1,1)$};
	\node (222) at (4, -7) {$\yng(2,2,2)$};
	\node (2211) at (0, -11) {$\yng(2,2,1,1)$};
	\node (21111) at (0, -14) {$\yng(2,1,1,1,1)$};
	\node (111111) at (0, -17) {$\yng(1,1,1,1,1,1)$};

    \draw (6) -- (51) -- (42) -- (411) -- (321) -- (3111) -- (2211) -- (21111) -- (111111);
    \draw (51) -- (33) -- (222) -- (2211);
    \draw (411) -- (222);
\end{tikzpicture}
\caption{The consumption ordering for $n=6$}
\end{figure}
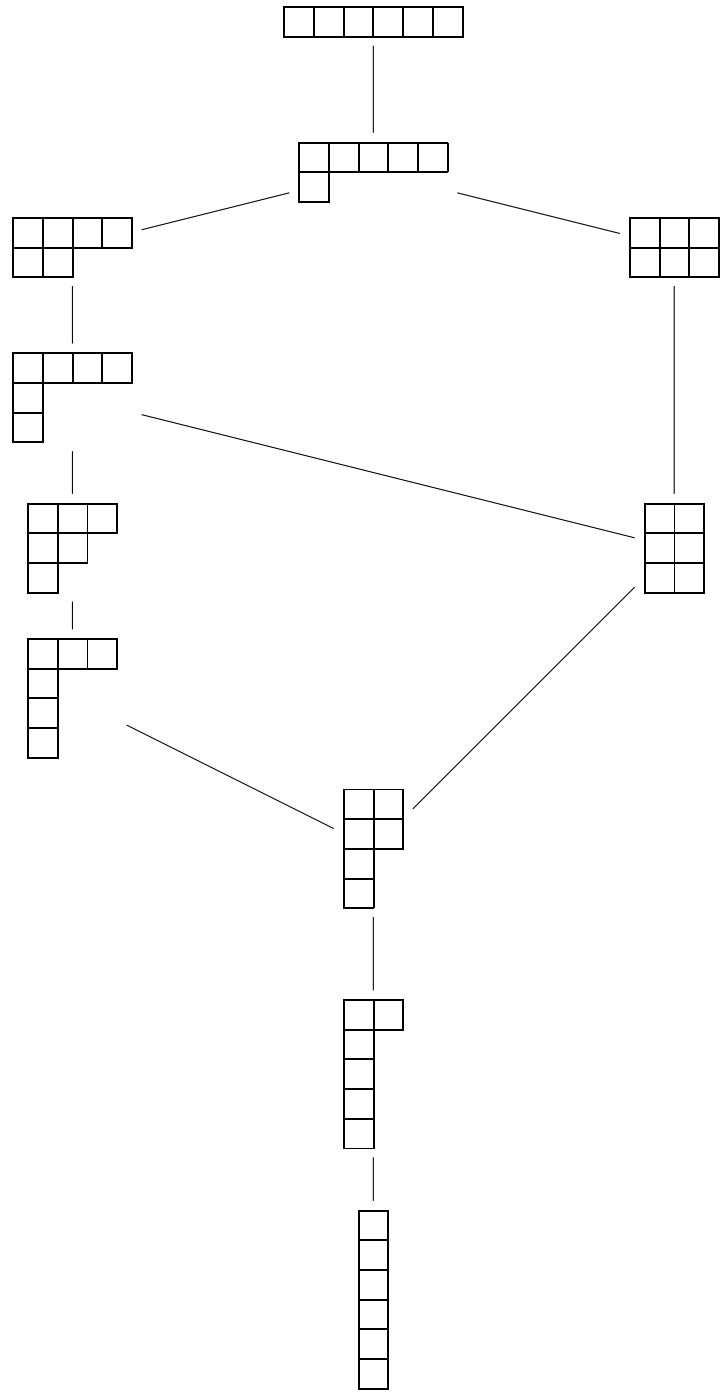

\section{Partitioning Cartesian powers of a set}
\label{gen}
In this section, we generalize the notion of distinguishing labelings in a different way.  Let $X=\{x_1, x_2, \ldots, x_n\}$ and let $V$ be the $\mathbb{C}-$vector space with $X$ as a basis.  Then by extending the action of $S_n$ linearly, $S_n$ naturally acts on $V$ by permuting the elements of $X$.  The guiding question we will explore is: 

\begin{question}
How can we go from a subgroup $G$ of $S_n$ to a subgroup $H\le G$ by restricting $S_n$ to its elements whose action preserve some partition?
\end{question}

For instance, let $n=5$ and $G=S_n$.  If we want $H$ to be the subgroup $S_3\times S_2$, where $S_3$ acts on the first three elements and $S_2$ on the last two, then we can partition the elements of $X$ by labeling $x_1, x_2, x_3$ with 1 and $x_4, x_5$ with 2.  With partitions of this form, we can clearly get all subgroups of the form $S_{\lambda_1}\times S_{\lambda_2}\times\cdots\times S_{\lambda_k}$ from $S_n$.

There are other ways of forming restrictions; for instance, we can add ``edges" to pairs of elements of $X$ and require that a permutation of $X$ preserves these edges.  This is, of course, the case of the automorphism group $\Gamma$ of corresponding graph.  Then the distinguishing number of the graph is simply the smallest length of a partition to go from $G=\Gamma$ to the trivial group.  Similarly, we may fix any subgroup of $S_n$ and ask the same question; this is distinguishing number with respect to arbitrary faithful group actions.

It turns out that we can in fact obtain all subgroups of $S_n$ through an appropriate partition.  However, it may not be a partition of the elements themselves.  For example, we can look at all pairs of elements.  Then $S_n$ acts on these pairs by acting on each entry separately.  If we partition these pairs and require that a permutation of the pairs preserve these partitions, we get a row subgroup which we may denote $P_{\lambda}$.  Then the corresponding subgroup of $S_n$ we obtain is $S_n\cap P_{\lambda}$.  We will show that through this method, we can in fact achieve every subgroup of $S_n$.

\subsection{Structure of automorphism groups of graphs}
Consider the action of $S_n$ on $S^2V = V\otimes V/(x_i\otimes x_j-x_j\otimes x_i)$ where $S_n$ permutes $X$ as usual.  

\begin{proposition}
\label{sym2}
Let $G$ be an undirected graph with vertex set $V=\{v_1, v_2, \ldots, v_n\}$.  Label the element $x_i\otimes x_j\in \operatorname{Sym}^2 V$ with the number of edges between $v_i$ and $v_j$ in $G$.  Then $\operatorname{Aut}(G)$ is canonically isomorphic to the subgroup of $S_n$ that preserves the labels on $\operatorname{Sym}^2 V$.

\begin{proof}
Every element of $\operatorname{Aut}(G)$ must preserve the labels on $\operatorname{Sym}^2 V$ in order for it to send edges to edges.  Conversely, if a subgroup of $S_n$ preserves the labels, then it is a permutation of vertices that sends edges to edges, as desired.
\end{proof}

\begin{corollary}
The case of simple undirected graphs follows when $x_i\otimes x_j$ is labeled $1$ if there is an edge between $v_i$ and $v_j$ and $0$ otherwise. 
\end{corollary}
\end{proposition}

We can also consider $V^{\otimes 2}$, which has basis $\{v_i\otimes v_j\}_{i, j = 1}^n$, and ask what subgroups labeling it can achieve.  This is equivalent to labeling $X^2$.  We have an analogue of Proposition~\ref{sym2}, which is proved similarly.

\begin{proposition}
\label{V2}
Let $G$ be a directed graph with vertex set $V=\{v_1, v_2, \ldots, v_n\}$.  Label the element $x_i\otimes x_j\in V^{\otimes 2}$ with the number of edges between $v_i$ and $v_j$ in $G$.  Then $\operatorname{Aut}(G)$ is canonically isomorphic to the subgroup of $S_n$ that preserves the labels on $V^{\otimes 2}$.
\end{proposition}

Note that we can recover the automorphism groups of undirected graphs by labeling the basis elements of $V^{\otimes 2}$ by simply labeling $(x_j, x_i)$ with the same label as $(x_i, x_j)$.

\begin{example}
We can obtain the automorphism group of $D_4$ through partitioning $\operatorname{Sym}^2V$ in the following way: label $(1,2), (2,3), (3,4), (4,1)$ with the label $1$ and all other basis elements with $2$.  We can also obtain it through partitioning $X^2$ in the following way: label $(1,2), (2,1), (2,3), (3,2), (3,4),$ $(4,3), (4,1), (1,4)$ with the label $1$ and all other elements with $2$. 
\end{example}

\begin{example}
We can obtain the cyclic group generated by $(1234)$ by partitioning $X^2$ in in the following way: label $(1,2) (2,3), (3,4), (4,1)$ with the label $1$ and all other elements with $2$.  This cannot be obtained as the automorphism group of an undirected graph.
\end{example}

\begin{proposition}
\label{char}
Let $G$ be a simple graph such that $\operatorname{Aut}(G)$ is 2-transitive.  Then $G$ is either $K_n$ for some $n$ or a graph with no edges.
\end{proposition}

\begin{proof}
If $G$ has any edge between two vertices $v_1$ and $v_2$, then by 2-transitivity there is some element of $\operatorname{Aut}(G)$ that sends it to any other pair of vertices.  Thus any pair of vertices must have an edge connecting them, so if $G$ has any edges it is $K_n$.
\end{proof}

\subsection{Automorphism groups of hypergraphs} 
By considering partitions of $X^k$, we can obtain more subgroups.  In particular, these subgroups can naturally be realized as automorphism groups of hypergraphs.

\begin{theorem}
\label{part}
We can obtain any subgroup of $S_n$ through a partition of $X^{n-1}$.
\end{theorem} 

\begin{proof}
Take any subgroup $H < S_n$.  Let $x = (x_1, x_2, \cdots, x_{n-1})$ and label every element of $Hx$ with 1.  Label every other element of $X^{n-1}$ with 2.  We claim that this labeling gives $H$.  It is clear that every element of $H$ preserves this labeling, because the elements labelled 1 form an orbit of $X$.  Thus it suffices to show that if $g$ preserves the labels, then $g\in H$.  Note that if $g(x)$ is labelled $1$, then it must be equal to $h(x)$ for some $h\in H$.  Then $g$ and $h$ must agree on where they send $x_1$, $x_2$, $\cdots$, $x_{n-1}$, and thus must agree on where they send $x_n$ too.  But this information completely determines the element $h$, so $g=h$ and thus $g\in H$, as desired.
\end{proof}

\begin{example}
$A_4$ can be realized with the following partition: 

$1: (1,2,3), (1,3,4), (1,4,2), (2,1,4), (2,3,1), (2,4,3), (3,1,2), (3,2,4), (4,1,3), (4,2,1), (4,3,2)$ and $2$ for everything else.
\end{example}

Motivated by Theorem~\ref{part}, we give the following definition.
\begin{definition}
[density] Let the density of a subgroup $G\le S_n$ be the least $k$ for which it can be realized as the elements of $S_n$ that preserve some partition of $X^k$.
\end{definition}

By Theorem~\ref{part}, we know that every subgroup $G$ of $S_n$ has density at most $n-1$.

\section{Distinguishing polynomial and distinguishing symmetric function}
In this section, we generalize the notion of distinguishing number to polynomials and symmetric function much like the chromatic number.  This leads to interesting questions described in the following section.

\begin{definition}
Given $G$ acting on $X$, let $f_X(n)$ be the number of distinguishing labelings of $X$ from the set $\{1, 2, \ldots, n\}$.  Then this function actually defines a polynomial $f_{G, X}(x)$, which we define to be the distinguishing polynomial of the group action of $G$ on $X$.
\end{definition}

\begin{theorem}
$f_{G, X}(x)$ is a polynomial of degree $n$.

\begin{proof}
Let $a_i$ be the number of distinguishing labelings of $X$ using precisely $i$ colors, for $1\le i\le n$.  Then $f_{G, X}(x)=\sum_{i=1}^na_i\binom{x}{i}$.
\end{proof}
\end{theorem}

As with the chromatic number, the distinguishing number is the smallest positive integer that is not a root of the distinguishing polynomial.

\begin{definition}
[DSF] Define the distinguishing symmetric function (DSF) of $G$ acting on $X$ as 
\[
Y_{G, X} = \sum_{(i_1, i_2, \ldots, i_n) \text{ is a distinguishing labeling for } X} x_{i_1}x_{i_2}\cdots x_{i_n}.
\]
\end{definition}

Clearly we can recover the distinguishing polynomial from the DSF by setting appropriate variables to 0 and 1.

\section{Further directions}
The techniques introduced in this paper lead to many new directions of research.  In this section we give an overview of these possibilities and state several questions and conjectures.  We begin with topics contained in the various sections of this paper and then discuss a possible applications to separate research on the distinguishing number. 

In Section~\ref{sn}, we proved that $S_n$ can act faithfully on a set $X$ with distinguishing number $n-1$ for arbitrarily large $n$.  Our construction made use of the fact that $A_n$ is a normal subgroup of $S_n$, so we could let $A_n$ be the stabilizer of a separate orbit alongside the $n$-element orbit on which $S_n$ acts naturally.  Furthermore, we showed that if $S_n$ acts with distinguishing number $n-1$, then it must act naturally on a set of $n$ elements.  This motivates the following problem.

\subsection{Distinguishing numbers for \texorpdfstring{$S_n$}{S_n}}
\begin{question}
Classify all actions of $S_n$ with distinguishing number $n-1$.
\end{question}

Furthermore, note that for $n>4$ the only non-trivial normal subgroup of $S_n$ is $A_n$, and that the bounds that we used to show the existence of a pseudoclique with $n$ elements are not tight.  This motivates the following conjecture.

\begin{conjecture}
There exists some $N$ such that for all $n>N$, $S_n$ cannot act with distinguishing number $n-2$.
\end{conjecture}

This can be strengthened to the following conjecture.

\begin{conjecture}
For every positive integer $k>1$, there exists some $N$ such that for  all $n>N$, $S_n$ cannot act with distinguishing number $n-k$.
\end{conjecture}

\subsection{Multiple labelings and obtaining subgroups}
In Section~\ref{gen1}, we considered the action of $\Gamma$ on $X^k$, given some action of $\Gamma$ on $X$.  We showed that this is equivalent to labeling each element of $X$ with a $k$-tuple of labels rather than just a single label.  In the case of $S_n$ acting on an $n$-element set via all possible permutations, we showed in Theorem~\ref{mul} that the distinguishing number is $\ceil{\sqrt[k]{n}}$.  We posit that a similar result might hold for general actions with the following question.

\begin{question}
Let $\Gamma$ act on $X$ with distinguishing number $r$.  Then does $\Gamma$ acts on $X^k$ with distinguishing number at most $\ceil{\sqrt[k]{r}}$?  When does equality hold?
\end{question}

We also considered the question of which subgroups of group actions could be obtained from labelings as the intersection of the group action and the subgroup of the symmetric group on the elements of the set that preserves the labels.  We showed in Theorem~\ref{abe} that all subgroups of an abelian group could be obtained.  We conjecture that this is in fact a necessary condition.  

\begin{conjecture}
If all subgroups of $\Gamma$ can be obtained through a labeling, then $\Gamma$ is abelian. 
\end{conjecture}

Note for instance that if $\Gamma$ is transitive on $X$, then any non-constant labeling will remove this property.  Thus we cannot obtain proper transitive subgroups of $\Gamma$. 

Recall that the Jordan-Holder theorem states that any two composition series of a group are equivalent.  Furthermore, recall that in Lemma~\ref{stab1}, the pointwise stabilizer of an orbit is normal.  By coloring the elements of an orbit with distinct colors and everything else a constant color, we thus obtain a normal subgroup of $\Gamma$.  This leads to the following problem.

\begin{problem}
Classify all group actions for which we can obtain an entire composition series of $\Gamma$ through refining labelings.
\end{problem}

Theorem~\ref{abe} shows us that when the group is abelian, this is always possible. 

\subsection{The consumption ordering of partitions}
In Section~\ref{poset}, we studied the different distinguishing partitions of group actions.  We saw that the consumption ordering is consistent with the dominance ordering in Theorem~\ref{dom}, but that the converse does not generally hold.  However, it turns out that for $n=2,3,5,7$, the consumption ordering is equivalent to the dominance ordering.  This prompts the following conjecture.

\begin{conjecture}
For all primes $p$, the consumption ordering on partitions of $p$ is equivalent to the dominance ordering.
\end{conjecture}

A related problem is to describe the cases in which $\lambda\ge \mu$ but $\lambda\not\ge_c\mu$.  The following dual conjecture may be solvable through clever constructions:

\begin{conjecture}
If $n$ is not prime, then the consumption ordering on partitions of $n$ is not equivalent to the dominance ordering.
\end{conjecture}

We also described a simple combinatorial principle that can be used to find some ordering relations in the consumption order.  It would be interesting if this could be made complete.

\begin{question}
Is there a combinatorial rule that can be applied to partitions that tells if one partition is greater than the other under the consumption ordering?
\end{question}

In order to specialize this to the case of automorphism groups of graphs, we need to make the corresponding adjustment to the definition of the consumption ordering.  Namely, we only consider subgroups of $S_n$ realizable as automorphisms of graphs on $n$ vertices in our definition of consumption.  We would like to describe the resulting poset.  For example, does the following result hold?

\begin{conjecture}
Considering only subgroups of $S_n$ realizable as the automorphism group of a graph, do we have $\lambda\ge_c\mu\Rightarrow \lambda\ge\mu$?
\end{conjecture}

By the fact that the consumption ordering defines a poset, we know that the set of subgroups that a partition is consumed by completely determines it.  The converse is certainly not true, for conjugate subgroups will consume the same partitions.  However, it is not known whether this is the only way in which the converse statement fails.  This leads to the following question.

\begin{question}
What can be said about the relationship between two subgroups of $S_n$ which consume the same partitions?
\end{question}

Another series of natural questions regards the existence of a ``best fit" subgroup that consumes each partition:

\begin{conjecture}
For each partition $\lambda\vdash n$, is there a subgroup that consumes $\lambda$ and no partition greater than $\lambda$ in the consumption ordering? 
\end{conjecture}

\begin{question}
Does the set of partitions a subgroup consumes always have a unique maximal element?
\end{question}

If the answer to the previous question is yes, then the following conjecture is equivalent to the previous conjecture; otherwise, it is a strengthening.

\begin{conjecture}
For each partition $\lambda\vdash n$, is there a subgroup that consumes $\lambda$, all partitions less than $\lambda$ in the consumption ordering, and no others? 
\end{conjecture}

In~\cite{Chan1}, Chan asked for a characterization of the following set:

\[
T_n=\{D_G([n]) ~|~ G \text{ is a transitive subgroup of } S_n\}.
\]

Since $D_G([n])$ is equal to the minimum length of the partitions of $[n]$ which $G$ consumes, we may ask the following stronger question.

\begin{question}
Characterize the set 
\[P_n=\{\lambda\vdash n ~|~ G \text{ is a transitive subgroup of } S_n \text{ that consumes } \lambda\}.\]
\end{question}

The following question is an enumerative one which arises from considering this poset.

\begin{question}
How many elements are in the union of the subgroups that consume a partition $\lambda$?  Equivalently, how many elements $g\in S_n$ have the property that if $g^i$ is the first power of $g$ in $a_{\lambda}$ where $a_{\lambda}$ is a Young projector, then $g^i$ is the identity?  
\end{question}

It is natural to combine this generalization with the two generalizations considered in this section and the previous one.  There are a myriad of interesting questions to be asked; here we present two open-ended ones.

\begin{question}
Can we relate the distinguishing partitions of $X^k$ to the distinguishing partitions of $X$?
\end{question}

\begin{question}
Which subgroups of $\Gamma$ can be obtained by requiring the elements to preserve some partition of a fixed type?
\end{question}

\subsection{Partitioning Cartesian powers of a set}

In Section~\ref{gen}, we looked at partitions of $X^k$ rather than simply partitions of $X$.  We saw that partitions of $\operatorname{Sym}^2(V)$ correspond to automorphism groups of undirected graphs.  This motivates the following question:

\begin{question}
Is there a meaningful way in which we can look at partitions of $\operatorname{Alt}^k(V)$ under this framework?
\end{question}  

In Proposition~\ref{char}, we show that if $\operatorname{Aut}(G)$ is 2-transitive, then $G$ is either $K_n$ or a graph with no edges.  This means that if we decompose the corresponding representation into irreducible representations, we will have more than 2 irreducible factors unless $\operatorname{Aut}(G)=S_n$.

\begin{question}
What can we say about the decomposition of the representation of $\operatorname{Aut}(G)$ into irreducible representations?
\end{question}

In particular, we conjecture that unless $G=S_n$, then the dimensions of the corresponding irreducible representations are generally small.

The following conjecture is motivated by the fact that only the symmetric and alternating groups are 6-transitive.
\begin{conjecture}
There exists some constant $c$ independent of $n$ such that the density of any subgroup $G\le S_n$ not equal to $A_n$ is at most $c$.
\end{conjecture}
By the previously stated fact, if this conjecture is true, one would expect $c$ to be not much larger than 6.

\subsection{Distinguishing symmetric function}
The first natural question to ask deals with how strong of an invariant the distinguishing polynomial and distinguishing symmetric function are. 

\begin{question}
Can the size of the group, or the group action on a set itself be recovered from the corresponding distinguishing polynomial or distinguishing symmetric function?
\end{question}

After computing the distinguishing symmetric function for graphs up to seven vertices, the author was surprised to find that, with four exceptions, all of which occur in graphs with six vertices, that the DSF is Schur-positive.  This motivates the following question. 

\begin{question}
Is there a class of graphs for which the distinguishing symmetric function can naturally be seen to be Schur-positive?  In particular, can it be realized as the character of a representation of a symmetric group in appropriate cases?
\end{question}

\subsection{Distinguishing extension numbers}

In~\cite{Ferrara}, the authors define the distinguishing extension number of a group action.  This is an extension of the distinguishing number in the sense that it differentiates some group actions which may have the same distinguishing number.  Here we describe a way to generalize this notion in the spirit of Section~\ref{poset}.

Given a group $\Gamma$ acting faithfully on a set $X$, a \textit{fixing set} $W$ is a subset of $X$ with trivial pointwise stabilizer with respect to the given action of $\Gamma$.  Then the \textit{distinguishing extension number} is defined to be the minimum $m$ such that of all fixing sets $W$ of size at least $m$,  any labeling $\phi: X\backslash W\rightarrow\{1, 2, \ldots, k\}$ can be extended to a distinguishing labeling of $X$. 

Consider extending the definition of fixing set to partitions in the following way.  Take a subset $W\subseteq X$ of size $m$ and a partition $\lambda\vdash m$.  Then if there is a labeling $\phi$ of $W$ with type $\lambda$ such that the only element of $\Gamma$ that preserves $\phi$ is the identity, then we say that $W$ is a $\lambda$-fixing set of $X$. 

If we restrict $\lambda$ to only contain parts of size 1, then we recover the notions behind the distinguishing extension number.  Thus we may use the definition of $\lambda$-fixing sets to pursue these ideas in greater generality. 

\section*{Acknowledgments}
This research was conducted at the University of Minnesota
Duluth REU, funded by NSF Grant 1650947 and NSA Grant H98230-18-1-0010.  The author would like to thank Joe Gallian and Jason Schuchardt for making many helpful comments on earlier drafts of the paper.  The author would also like to  thank his advisors Levent Alpoge, Aaron Berger, and Colin Defant, his peers, and Joe again for creating a stimulating and supportive work environment at Duluth.

\end{document}